\documentclass[onecolumn]{IEEEtran}
\usepackage[utf8]{inputenc}
\usepackage{amsmath}
\usepackage{mathtools}
\usepackage{amsthm}
\usepackage{amsfonts}
\usepackage{graphicx}
\usepackage{tikz,pgfplots}
\pgfplotsset{compat=1.14}
 
\newtheorem{remark}{Remark}
\newtheorem{model}{Model}
\newtheorem{definition}{Definition}
\newtheorem{theorem}{Theorem}

\newtheorem{corollary}{Corollary}
\newtheorem{lemma}{Lemma}

\usepackage{draftwatermark}
\SetWatermarkText{DRAFT}
\SetWatermarkScale{1}

\title{Gradient Formulation for the Stability of DC-Microgrids}
\author{Alejandro Garces}
\date{August 2018}

\begin{document}

\maketitle
\begin{abstract}
\textcolor{black}{
This paper presents a non-linear stability analysis for dc-microgrids in both, interconnected mode and island operation with primary control.  The proposed analysis is based on the fact that the dynamical model of the grid is a gradient system generated by a strongly convex function. The stability analysis is thus reduced to a series of convex optimization problems. The proposed method allows to: i) demonstrate the existence and uniqueness of the equilibrium ii) calculate this equilibrium numerically iii) give conditions for global stability using a Lyapunov function iv) estimate the attraction region.  Previous works only address one of these aspects. Numeric calculations performed in cvx and simulations results in Matlab complement the analysis and demonstrate how to use this theoretical results in practical problems.  
}
\end{abstract}


\section{Introduction}

\IEEEPARstart{DC}{microgrids} and dc distribution are promising technologies for integrating solar panels, batteries and fuel-cells among other components that operate in dc.   Most of these components are integrated through a power electronic converter that controls a constant power resulting in a non-linear circuit \cite{review_dcmicrogirds}. In particular,  constant power loads can introduce a negative resistance effect which in turns generates transient stability problems.  Rigorous methodologies for stability analysis are required in this new context.

Several stability methodologies have been proposed for systems with ad-hoc controls \cite{adhoc}.  However, the conventional droop control with a constant power model is the most common approach for control and stabilization of microgrids \cite{DC_Parte1}. Most of the stability studies on this type of controls are based on linearization (i.e small signal stability).  Transient stability in these generalized models is still a challenge due to their non-linear behavior; even finding the equilibrium point can be a challenge.

This paper proposes a methodology for transient stability of dc-microgrids based on the study of gradient systems.  Although these type of systems have a rich and general theory, we are interested in a particular type, namely, gradient systems with a strongly convex function (we will call these as strongly-convex-gradient systems). \textcolor{black}{We demonstrate the existence and uniqueness of the equilibrium as well as the conditions for global stability. In addition, we show a simple method for calculating this equilibrium and also estimating the region of attraction.  Being a convex problem, we can guarantee convergence of the algorithms.  Our model is simple enough to be tractable computationally but showing the main interaction between components of the nonlinear circuit.  Just as in the case of the second order model for power systems applications}. The proposed analysis shows a surprising connection between dynamical systems and convex analysis.  This connection is explored from a practical point of view, since the stability analysis is transformed into a series of convex optimization problems that can be solved numerically.  To the best of the author's knowledge, there is not applications of the approach proposed in this paper.

The existence of the equilibrium was analyzed in \cite{koguiman} for one constant power load and generalized in \cite{romeo_existencia_equilibrio} for several loads.   A different approach was presented in \cite{yo_elsevier} and \cite{yo_tps} based on the convergence of the power flow.   In the later methods, it was demonstrated that conventional algorithms such as Gauss and Newton's methods converge to a unique equilibrium point under well defined conditions. 

From the stability point of view, several studies have been presented for the small signal case \cite{small}\cite{small2}\cite{small3}. However, transient stability studies are required in order to increase the accuracy of the study and consider the non-linear behavior of the grid. A recent review of transient stability analysis in microgrids can be found in \cite{transient_review}.\textcolor{black}{That review showed the necessity of systematic methods for large-scale stability analysis and reduced order models of the grid. Our method fulfills these conditions}.   	  In \cite{stab} a stabilization method was proposed for a dc-microgrid \textcolor{black}{in which all the terminals were connected to the same bus-bar with only one equivalent constant power load. Droop controls were considered only on the sources. Our method consider the topology of the grid with different constant power loads and droop control in both the constant voltage and constant power terminals.} In fact, constant power loads have been the main concern of recent stability analysis such as \cite{cpl1} and \cite{estabilidad_sdp}. However, most of these studies are developed for ad-hoc controls requiring  a detailed model of the converter. \textcolor{black}{None of these approaches reveals the gradient characteristics of the model.} 

\textcolor{black}{The use of convex analysis is also a contribution of this paper.  Although this type of analysis is usually consider in linear matrix inequalities, that type of analysis is linear whereas the method presented here is nonlinear}.  In \cite{estabilidad_sdp} a semidefinite programming methodology was proposed by formulating a Lure problem with quadratic bounds. That formulation allowed to estimate the region of attraction in grids with constant power loads.  However the topology of the grid was limited to a unique bus-bar and the analysis is basically linear.

The rest of the paper is organized as follows:  Section II presents the dynamical model of the dc-microgrid considering constant power terminals.  Section III describes the stability analysis based on the use of gradient systems with strongly convex functions. Section IV explain how the stability problem is transformed into convex optimization models which allow to determine the attraction region, equilibrium point and under-voltage limit. Simulation results are presented in Section V followed by conclusions, appendix and references.

\section{Problem definition}

Let us consider a dc-microgrid with droop control that is expected to operate whether grid-connected or in island-mode.  The master node is represented by $0$ and maintains a constant voltage $v_0$; the rest of the nodes are represented by $N=\left\{1,2,\cdots n \right\}$ and maintain a constant power which can be positive or negative. The grid is purely restive and is represented by the admittance matrix which includes linear loads eliminated by a kron's reduction \cite{kron}. The model of each constant power terminal is depicted in Fig \ref{fig:cpl}. It includes the capacitive effect of the converter and the droop control. This model is widely used in different applications including dc-microgrids \cite{danilo} and multiterminal HVDC transmision \cite{hvdc}.    

\begin{figure}[tb]
    \centering
    \footnotesize
    \begin{tikzpicture}[x=0.8mm,y=0.8mm,thick,blue!70!green]
    \draw[gray,fill=gray!10] (-8,-15) rectangle +(70,30);
    \draw[gray,fill=gray!10] (70,-15) rectangle +(20,70);
    \node[text width=30] at (82,15) {Grid};
    \draw[-latex] (33,10) node[above] {$p_i$} -- +(0,-8);
    \draw (33,0) circle (2);
    \draw[-latex] (21,0) -- +(10,0);
    \draw (11,-5) rectangle +(10,10);
    \draw [-latex] (5,0) node[left] {$\Delta v_i$} -- +(6,0);
    \node at (17,0) {$k_i$};
    \draw[-latex] (35,0) -- (45,0);
    \draw (45,0) -- +(2.5,5) -- +(5,0) -- +(2.5,-5) -- cycle;
    \draw[-latex] (47.5,-3) -- +(0,6);
    \draw (47.5,5) |- +(22.5,8);
    \draw (47.5,-5) |- +(22.5,-8);
    \draw (55,13) -- +(0,-12);
    \draw (55,-13) -- +(0,12);
    \draw (53,1) -- +(4,0);
    \draw (53,-1) -- +(4,0);
    \node at (40,4) {$\frac{1}{v_i}$};
    \node at (63,0) {$c_i\frac{dv_i}{dt}$};
    \node at (25,17) {Slave node with primary control};
    \draw[gray, fill=gray!10] (50,50) rectangle +(10,8);
    \node at (53,47) {Switch};
    \draw (10,40) circle (3);
    \node at (10,40) {$\pm$};
    \node at (17,40) {$v_0$};
    \node at (27,50) {Master node};
    \draw (10,43) |- +(43,10) -- +(46,13);
    \draw (55,53) -- +(15,0);
    \draw (10,37) |- +(60,-10);
    \end{tikzpicture}
    \caption{Shcematic representation of the dc-microgrid}
    \label{fig:cpl}
\end{figure}
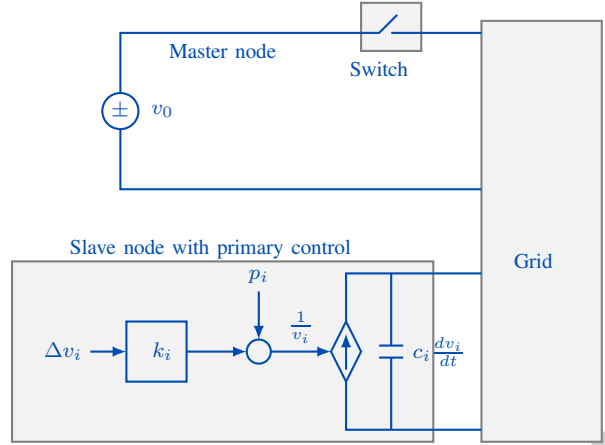

The dynamics of the grid can be represented as
follows:
\begin{equation}
c_i \frac{dv_i}{dt} = \frac{p_i+k_i(1-v_i)}{v_i}-g_{0i}v_0 - \sum_{i=1}^{n} g_{ij} v_j 
\label{eq:model}
\end{equation}
where $v_i$ are nodal voltages, $g_{ij}$ are entries of the admittance matrix and $c_i$ is the capacitance of the converter. Notice this is a non-linear dynamical system due to the presence of constant power devices.  The model is the same for grid connected or island operation.  The only difference is that under island operation, the switch is opened and hence $g_{i0}=0$. 

On the other hand, each converter is equipped with a droop control with constant $k_i$.   We assume that $c_i>0$ and $k_i\geq 0$ meaning that we can have terminals that do not participate in the primary control ($k_i$ can be zero). In addition, the graph that represents the grid is connected and hence $G=[g_{ij}]\succ 0$ (i.e it is positive definite). We allow transients generated by disconnection of the master node or abrupt changes in generation or demand.  However, we assume the grid is not under short circuit and hence $v_i>0$ (our model would produce an infinite current for a shorcircuit in the capacitor).  

It is important to notice that this model is general enough and allows to represent different components according to their functionality.  For example, a solar panel is just a constant power device with $p_i>0$ while a constant power load is the same model but with $p_i<0$.   As we will demonstrate below, constant power generation does not create big stability problems being the constant power loads the main source of instabilities. A simple justification of this fact, is that constant power loads generates a negative resistance effect that induce instability.

\section{Stability Analysis}

\subsection{Convex analysis and gradient systems}

We are interested in using convex optimization for analyzing stability in a particular type of dynamical systems, namely, gradient systems with strongly convex functions.  Therefore, we will use some results from convex optimization restricting our analysis to continuously differentiable functions, in order to be used as Lyapunov candidates.  Following this idea, let us start by this classical result 

\begin{theorem}
Consider a non-empty convex set $\Omega\in\mathbb{R}^n$ and a twice differentiable function $W:\Omega\rightarrow \mathbb{R}$  such that
\begin{equation}
    \frac{\partial ^2 W}{\partial x^2} \succeq \mu I_N
\end{equation}
where $\mu$ is a real number such that $\mu>0$, $\partial ^2 W/\partial x^2$ is the Hessian matrix of $W$ and $I_N$ is the identity matrix of size $n$. We say that $W$ is a strongly convex; for this type of function there exists a global minimum  $\tilde{x}\in\Omega$ and this minimum is unique.
\end{theorem}
\begin{proof}
See \cite{nemirovski}
\end{proof}

\begin{remark}
There is a more general theorem which includes strictly convex functions (see \cite{nemirovski} for more details).  However, Theorem 1 is general enough for our dynamical problem and avoids the proliferation of unnecessary definitions and technicalities. 
\end{remark}

\begin{remark} It is well known that convex functions have a global optimum.  However, for the case strongly convex functions, this optimum is not only global but also unique. \end{remark}

Let us connect this result with the dynamical system theory by proposing the following definition

\begin{definition}
A strongly-convex-gradient system is a dynamical system that can be represented as
\begin{equation}
    M(x) \frac{dx}{dt} = -\frac{\partial W}{\partial x}
    \label{eq:gradient_system}
\end{equation}
where $M\succ 0$ (i.e positive definite) and $\partial H/\partial x$ represents the gradient of a strongly convex function $W:\Omega\in\mathbb{R}^n\rightarrow \mathbb{R}$.
\end{definition}

With this simple definitions we can present our first result:
\begin{theorem}
 Consider a strongly-convex-gradient system described by  (\ref{eq:gradient_system}). Then, there is a unique equilibrium point $\tilde{x}\in\Omega$ and this equilibrium is asymptotically stable.  In addition $\Omega$ is an estimation of the atraction region.
\end{theorem}

\begin{proof}  
Consider the gradient system (\ref{eq:gradient_system}) with a strongly convex function $W$.  The equilibrium point is given by the points in which $\partial W/\partial x = 0$ which correspond to the global minimum of $W$.  This minimum is also unique due to Theorem 1, guaranteeing the existence and uniqueness of the equilibrium point.  For the stability analysis, consider a Lypunov function $\mathcal{V}(x)=W(x)-W(\tilde{x})$ which evidently fulfills the conditions for stability, namely  $\mathcal{V}(\tilde{x})=0$, $\mathcal{V}(x)>0,\; \forall x\in\Omega-\left\{\tilde{x}\right\}$ and 
\begin{align}
    \frac{d\mathcal{V}}{dt} &= -\left(\frac{\partial W}{\partial x}\right)^T M(x)^{-1} \left(\frac{\partial W}{\partial x}\right) \leq 0
\end{align}
Asymptotic stability is directly demonstrated by invoking  LaSalle's invariant principle principle.

\end{proof} 

This simple result allows to transform the problem of stability into a convex optimization problem.  This is an advantage since we can calculate numerically the equilibrium point and estimate the region of attraction as will be demonstrated in the next sections.

\subsection{Dc-microgrids as gradient systems}

Dc-microgrids can be represented as strongly-convex gradient systems.  Let us formalize this by the following lemma that can be demonstrated by simple substitution and corresponding derivation:

\begin{lemma}
The system given by (\ref{eq:model}) can be represented as a gradient system with gradient $W$ given by 
\begin{equation}
\begin{split}
    W(v) = -(p_i+k_i) \ln(v_i) + \sum_{i=1}^{n} g_{oi}v_0v_i + \\ \frac{1}{2}\sum_{i=1}^{n}\sum_{j=1}^{n} g_{ij}v_i v_j + \frac{1}{2} \sum_{i=1}^{n} k_iv_i^2
\end{split}
\end{equation}

and $M=diag(c_i)\succ 0$.  In addition, the Hessian matrix of $W$ is given by
\begin{equation}
\frac{\partial^2 W}{\partial v^2} =
(P+K)X + G + K
\label{eq:hesiana}
\end{equation}
where $P=diag(p_i)$, $K=diag(k_i)$ and $X=diag(1/v_i^2)$
\end{lemma}
In order to demonstrate stability, it is enough to establish that exists a convex set $\Omega$ in which $W$ is strongly convex. This criteria is presented in the following theorem which constitute the main theoretical result of this paper: 

\begin{theorem}
A dc-microgrid with the conditions of Lemma 1 has an asymptotically stable equilibrium point if there exists a non-empty set containing the equilibrium $\tilde{x}$ which is given by
\begin{equation}
\Omega_x = \left\{\begin{array}{rl}
    (P+K)X+G + K &\succeq \mu I_N \\ 
    X &= diag(x) \\
    x_{i}\geq\; \tilde{x}_i &\geq 0 
    \end{array}\right\}
    \label{eq:set}
\end{equation}
with $\mu>0$. This set is an estimation of the region of attraction.
\end{theorem}

\begin{proof}  
It is easy to see from Lemma 1 that the function $W$ is strongly convex  if the Hessian given by (\ref{eq:hesiana}) is such that $\partial^2 W/\partial v \succeq \mu I_N$. This is evidently true if the set (\ref{eq:set}) exists.  This is an estimation of the region of attraction since it is an invariant set.  
\end{proof} 

\begin{remark}
Recall that one of our main assumptions is that no capacitor is in short-circuit meaning that $v_i>0$.  Therefore, the change of variables $x_{i}=1/v_i^2$ is a bijection, i.e we can obtain directly the value of $v_i$ from the value of $x_{i}$ and vice versa.  We define two spaces $\Omega_x=\left\{x\in\mathbb{R}^n_+\right\}$ and $\Omega_v=\left\{v\in\mathbb{R}^n_+\right\}$ in order to simplify our nomenclature.

\end{remark}

This theorem, allows a simple methodology for analyzing stability in dc-microgrids using convex optimization. For example, a direct result is obtained for the generation case:   
\begin{corollary}
For the generation case  $(p_i>0)$ there is a unique equilibrium point and the region of attraction is the entire first quadrant (i.e $v_i>0$)
\end{corollary}

\begin{proof}
Notice that for the generation case $p_i>0$ and hence we can increase $X$ as large as desired in order to fulfill the generalized inequality.  The set of constraints is fulfilled even for $x_i\rightarrow \infty$ which correspond to the entire fisrt quadrant in both, the space $x$ and the space in $v$.
\end{proof}
This corollary only formalize a result that is well known for practical applications: generation case is not an stability issue being the constant power loads the real problem.  In the next subsections, we analyze the general case with both constant power generation and constant power loads.

\section{Stability analysis using convex optimization}

Theorem 3 allows to define different convex optimization models for the stability analysis.  In the following subsections we present different optimization models according to the desired analysis.

\subsection{Equilibrium point}

The existence and uniqueness of the equilibrium point on dc microgrids has been analyzed previously for different authors (see for example \cite{yo_elsevier} and \cite{yo_tps}).  Here, we use Theorem 3 and define a simple optimization model

\begin{model}[equilibrium point]
\begin{equation}
\underset{v_i}{\min} \; W(v)
\end{equation}
\end{model}

Hence, the equilibrium point of (\ref{eq:model}) can be obtained as $\tilde{v} = \text{argmin}\left\{W(v)\right\}$, given that $\Omega_x$ exists and contains the point $\tilde{x}_i=1/\tilde{v}_i^2$. The optimization problem can be solved by gradient or Newton's methods guaranteeing uniqueness and convergence if $W$ is  strongly-convex.   The gradient method has a linear convergence while the newton method has a quadratic convergence (see \cite{nemirovski} for more details).

\subsection{Region of attraction}

For the general case, the region of attraction can be estimated by maximizing the hypervolume of $\Omega_x$. However, the shape of $\Omega_x$ is a spectrahedron which can be quite complicated for being used in practical applications. For the sake of simplicity, we inscribe a hypercube $\Gamma_x$ inside $\Omega_x$ defined by $\Gamma_x=\left\{ x\in\Omega_x: x_i\geq \alpha\right\}$. Then, our objective is to maximize the hypervolume of $\Gamma_x$ by maximizing $\alpha$.  Therefore, an inner approximation of the region of attraction can be obtained by solving the following optimization model

\begin{model}[region of attraction]
\begin{align}
    \underset{\alpha,x_i}{\max} \; & \alpha \\
    (P+K)X+G + K &\succeq \mu I_N \\ 
    X &= diag(x) \\
    x_{i}&\geq \tilde{x}_i \\
    x_i &\geq \alpha
\end{align}
\end{model}

Notice this is a convex optimization problem. If the result is a positive $\alpha$ then the equilibrium is asymptotically stable and the values of $x_i$ define the limits of the hyper cube that represents an estimation of the region of attraction (in Section VI we show that this estimation is very accurate in practice).   

It is possible to return to the original values by using the transformation from $\Omega_x$ to $\Omega_v$. In practice, this set can be considered as a limit for a voltage stability analysis and as an under-voltage protection, i.e the component must be disconnected for voltages lower than these voltages.  Notice that Model 2 does not depend on the voltage $v_0$ and hence it can be used during voltage sags in the main grid (we will talk more about this issue in the results section).

\section{Results}

The proposed methodology was evaluated in the 10-node dc-microgrid depicted in Fig \ref{fig:test}. Nominal values of the grid are 380V/1kW and each line segment has a resistance of $1.5m\Omega/m$.   The rest of the parameters are given in Table \ref{tab:parametros}.  

Before considering the entire grid, led use study the case in which only load $p_4$ and generation $p_7$ are connected.  The motivation of such a basic case, is that the dynamical model is in $\mathbb{R}^2$ allowing graphical representations (see \cite{yo_doscpl} for a nullcline analysis on the plane for this simplified case). 

\begin{table}[]
    \centering
    \caption{Parameters of the constant power terminals and results of each model}
    \label{tab:parametros}
    \begin{tabular}{|c|c|c|c||c|c|}
    \hline
         Node & $k_i (\%)$ & $p_i (kW)$ & $c_i (\mu F)$ & Model 1 & Model 2 \\
    \hline\hline     
    1 & 1 & -80 & 690 & 0.9893 & 0.0958  \\
    2 & 1 & -70 & 560 & 0.9851 & 0.0958  \\
    3 & 1 &  80 & 630 & 0.9917 & 0.0001  \\ 
    4 & 1 & -50 & 500 & 0.9881 & 0.0710  \\
    5 & 1 & -70 & 690 & 0.9954 & 0.0958  \\
    6 & 1 &  80 & 640 & 1.0047 & 0.0001  \\
    7 & 1 & 100 & 600 & 1.0129 & 0.0001  \\
    8 & 1 &  50 & 630 & 1.0026 & 0.0001  \\
    9 & 1 & -90 & 690 & 0.9973 & 0.0818  \\
    \hline     
    \end{tabular}
\end{table}

\begin{figure}[tb]
\footnotesize
\centering
\begin{tikzpicture}[x=1mm, y = 1mm, thick]

\node  at (0,0) [draw, circle]    (N1) {0};
\node  at (0,-20) [draw, circle]   (N2) {1};
\node at (-20,-20) [draw, circle]   (N3) {2};
\node at (20,-20) [draw, circle]   (N4) {3};
\node at (40,-20) [draw, circle]  (N5) {4};
\node at (0,-40) [draw, circle]  (N6) {5};
\node at (0,-60) [draw, circle]  (N7) {6};
\node at (0,-80) [draw, circle]  (N8) {7};
\node at (20,-60) [draw, circle]  (N9) {8};
\node at (-20,-80) [draw, circle]  (N10) {9};

\draw (N1) -- (N2) -- (N6) -- (N7) -- (N8) -- (N10);
\draw (N2) -- (N4) -- (N5);
\draw (N2) -- (N3);
\draw (N7) -- (N9);

\draw[blue!50!green, fill = blue!50!green!30] (-3,5) rectangle +(6,6);
\draw[blue!50!green] (-3,5) -- +(6,6);
\node at (-1,9.5) {$\approx$};
\node at (1.5,7) {$=$};
\draw[blue!50!green] (N1) -- (0,5);
\node at (15,9) [text width = 50, blue!50!green] {Connection to the AC grid};
\node at (15,0) [text width = 50] {constant voltage 1pu};

\draw[blue!50!green] (N3) -- +(0,10);
\draw[blue!50!green] (N7) -- +(-10,10);
\draw[blue!50!green] (N4) -- +(0,10);
\draw[blue!50!green] (N6) -- +(10,0);
\draw[blue!50!green] (N9) -| +(10,10);
\draw[blue!50!green] (N8) -| +(10,10);
\draw[blue!50!green] (N2) -- +(-10,-10);
\draw[blue!50!green] (N10) -- +(0,10);
\draw[blue!50!green] (N5) -- +(0,10);

\node[rotate=90] at (-2,-10) {50m};
\node at (-10,-18) {15m};
\node at ( 10,-18) {20m};
\node at (30,-18) {18m};
\node[rotate=90] at (-2,-30) {23m};
\node[rotate=90] at (-2,-50) {17m};
\node[rotate=90] at (-2,-70) {21m};
\node at (10,-58) {13m};
\node at (-10,-78) {15m};

\begin{scope}[ very thick, scale=0.5, xshift=120, yshift=-310]
	\fill[gray!50] (0,0) -- +(5,5) -- +(32,7) -- +(21,0) -- cycle;
	\fill[left color = green!40!blue,right color=green!40!blue!60] (0,0) -- +(5,21) -- +(26,21) -- +(21,0) -- cycle;
	\foreach \x in {0,3,...,21}  \draw[-,white] (\x,0) -- +(5,21);
	\foreach \y in {0,3,...,21}  \draw[-,white] (0.25*\y,\y) -- +(21,0);	
\end{scope}

\begin{scope}[ very thick, scale=0.5, xshift=30, yshift=-430]
	\fill[gray!50] (0,0) -- +(5,5) -- +(32,7) -- +(21,0) -- cycle;
	\fill[left color = green!40!blue,right color=green!40!blue!60] (0,0) -- +(5,21) -- +(26,21) -- +(21,0) -- cycle;
	\foreach \x in {0,3,...,21}  \draw[-,white] (\x,0) -- +(5,21);
	\foreach \y in {0,3,...,21}  \draw[-,white] (0.25*\y,\y) -- +(21,0);	
\end{scope}

\begin{scope}[ very thick, scale=0.5, xshift=-110, yshift=-310]
	\fill[gray!50] (0,0) -- +(5,5) -- +(32,7) -- +(21,0) -- cycle;
	\fill[left color = green!40!blue,right color=green!40!blue!60] (0,0) -- +(5,21) -- +(26,21) -- +(21,0) -- cycle;
	\foreach \x in {0,3,...,21}  \draw[-,white] (\x,0) -- +(5,21);
	\foreach \y in {0,3,...,21}  \draw[-,white] (0.25*\y,\y) -- +(21,0);	
\end{scope}

\begin{scope}[ very thick, scale=0.5, xshift=60, yshift=-80]
	\fill[gray!50] (0,0) -- +(5,5) -- +(32,7) -- +(21,0) -- cycle;
	\fill[left color = green!40!blue,right color=green!40!blue!60] (0,0) -- +(5,21) -- +(26,21) -- +(21,0) -- cycle;
	\foreach \x in {0,3,...,21}  \draw[-,white] (\x,0) -- +(5,21);
	\foreach \y in {0,3,...,21}  \draw[-,white] (0.25*\y,\y) -- +(21,0);	
\end{scope}

\begin{scope}[gray, xshift=-50, yshift=-30, scale=0.3]
\draw[black,  fill=gray] (-22,2) -- +(4,15) -- +(40,15) -- +(44,0) -- cycle;
\draw[black,  fill=gray] (-8,17) -- +(0,3) -- +(4,3) -- +(4,0) -- cycle;
\draw[gray!30,top color=gray!30, bottom color=white] (-20,1) rectangle +(40,-2);
\draw[gray!30,top color=gray!30, bottom color=white] (-20,-1) rectangle +(40,-2);
\draw[gray!30,top color=gray!30, bottom color=white] (-20,-3) rectangle +(40,-2);
\draw[gray!30,top color=gray!30, bottom color=white] (-20,-5) rectangle +(40,-2);
\draw[gray!30,top color=gray!30, bottom color=white] (-20,-7) rectangle +(40,-2);
\draw[gray!30,top color=gray!30, bottom color=white] (-20,-9) rectangle +(40,-2);
\draw[gray!30,top color=gray!30, bottom color=white] (-20,-11) rectangle +(40,-2);
\draw[gray!30,top color=gray!30, bottom color=white] (-20,-13) rectangle +(40,-2);
\draw[black, fill=blue!70!green!50] (-17,-9) rectangle +(6,7);
\draw[black] (-14,-9) -- +(0,7);
\draw[black] (-17,-5.5) -- +(6,0);
\draw[black, fill=blue!70!green!50] (-8,-9) rectangle +(6,7);
\draw[black] (-5,-9) -- +(0,7);
\draw[black] (-8,-5.5) -- +(6,0);
\draw[black,fill=blue!70!green!50] (11,-9) rectangle +(6,7);
\draw[black] (14,-9) -- +(0,7);
\draw[black] (11,-5.5) -- +(6,0);
\draw[black, fill=gray] (1,-15) rectangle +(7,14);
\draw[gray!30,fill] (7,-8) circle (0.5);
\draw[black] (-20,1) rectangle +(40,-16);
\draw[black!70, fill] (-22,1) rectangle +(44,1);
\draw[black!70, fill] (0,-15.5) rectangle +(9,1);
\end{scope}
\begin{scope}[gray, xshift=40, yshift=-120, scale=0.3]
\draw[black,  fill=gray] (-22,2) -- +(4,15) -- +(40,15) -- +(44,0) -- cycle;
\draw[black,  fill=gray] (-8,17) -- +(0,3) -- +(4,3) -- +(4,0) -- cycle;
\draw[gray!30,top color=gray!30, bottom color=white] (-20,1) rectangle +(40,-2);
\draw[gray!30,top color=gray!30, bottom color=white] (-20,-1) rectangle +(40,-2);
\draw[gray!30,top color=gray!30, bottom color=white] (-20,-3) rectangle +(40,-2);
\draw[gray!30,top color=gray!30, bottom color=white] (-20,-5) rectangle +(40,-2);
\draw[gray!30,top color=gray!30, bottom color=white] (-20,-7) rectangle +(40,-2);
\draw[gray!30,top color=gray!30, bottom color=white] (-20,-9) rectangle +(40,-2);
\draw[gray!30,top color=gray!30, bottom color=white] (-20,-11) rectangle +(40,-2);
\draw[gray!30,top color=gray!30, bottom color=white] (-20,-13) rectangle +(40,-2);
\draw[black, fill=blue!70!green!50] (-17,-9) rectangle +(6,7);
\draw[black] (-14,-9) -- +(0,7);
\draw[black] (-17,-5.5) -- +(6,0);
\draw[black, fill=blue!70!green!50] (-8,-9) rectangle +(6,7);
\draw[black] (-5,-9) -- +(0,7);
\draw[black] (-8,-5.5) -- +(6,0);
\draw[black,fill=blue!70!green!50] (11,-9) rectangle +(6,7);
\draw[black] (14,-9) -- +(0,7);
\draw[black] (11,-5.5) -- +(6,0);
\draw[black, fill=gray] (1,-15) rectangle +(7,14);
\draw[gray!30,fill] (7,-8) circle (0.5);
\draw[black] (-20,1) rectangle +(40,-16);
\draw[black!70, fill] (-22,1) rectangle +(44,1);
\draw[black!70, fill] (0,-15.5) rectangle +(9,1);
\end{scope}

\begin{scope}[gray, xshift=100, yshift=-30, scale=0.3]
\draw[black,  fill=gray] (-22,2) -- +(4,15) -- +(40,15) -- +(44,0) -- cycle;
\draw[black,  fill=gray] (-8,17) -- +(0,3) -- +(4,3) -- +(4,0) -- cycle;
\draw[gray!30,top color=gray!30, bottom color=white] (-20,1) rectangle +(40,-2);
\draw[gray!30,top color=gray!30, bottom color=white] (-20,-1) rectangle +(40,-2);
\draw[gray!30,top color=gray!30, bottom color=white] (-20,-3) rectangle +(40,-2);
\draw[gray!30,top color=gray!30, bottom color=white] (-20,-5) rectangle +(40,-2);
\draw[gray!30,top color=gray!30, bottom color=white] (-20,-7) rectangle +(40,-2);
\draw[gray!30,top color=gray!30, bottom color=white] (-20,-9) rectangle +(40,-2);
\draw[gray!30,top color=gray!30, bottom color=white] (-20,-11) rectangle +(40,-2);
\draw[gray!30,top color=gray!30, bottom color=white] (-20,-13) rectangle +(40,-2);
\draw[black, fill=blue!70!green!50] (-17,-9) rectangle +(6,7);
\draw[black] (-14,-9) -- +(0,7);
\draw[black] (-17,-5.5) -- +(6,0);
\draw[black, fill=blue!70!green!50] (-8,-9) rectangle +(6,7);
\draw[black] (-5,-9) -- +(0,7);
\draw[black] (-8,-5.5) -- +(6,0);
\draw[black,fill=blue!70!green!50] (11,-9) rectangle +(6,7);
\draw[black] (14,-9) -- +(0,7);
\draw[black] (11,-5.5) -- +(6,0);
\draw[black, fill=gray] (1,-15) rectangle +(7,14);
\draw[gray!30,fill] (7,-8) circle (0.5);
\draw[black] (-20,1) rectangle +(40,-16);
\draw[black!70, fill] (-22,1) rectangle +(44,1);
\draw[black!70, fill] (0,-15.5) rectangle +(9,1);
\end{scope}

\begin{scope}[gray, xshift=-50, yshift=-190, scale=0.3]
\draw[black,  fill=gray] (-22,2) -- +(4,15) -- +(40,15) -- +(44,0) -- cycle;
\draw[black,  fill=gray] (-8,17) -- +(0,3) -- +(4,3) -- +(4,0) -- cycle;
\draw[gray!30,top color=gray!30, bottom color=white] (-20,1) rectangle +(40,-2);
\draw[gray!30,top color=gray!30, bottom color=white] (-20,-1) rectangle +(40,-2);
\draw[gray!30,top color=gray!30, bottom color=white] (-20,-3) rectangle +(40,-2);
\draw[gray!30,top color=gray!30, bottom color=white] (-20,-5) rectangle +(40,-2);
\draw[gray!30,top color=gray!30, bottom color=white] (-20,-7) rectangle +(40,-2);
\draw[gray!30,top color=gray!30, bottom color=white] (-20,-9) rectangle +(40,-2);
\draw[gray!30,top color=gray!30, bottom color=white] (-20,-11) rectangle +(40,-2);
\draw[gray!30,top color=gray!30, bottom color=white] (-20,-13) rectangle +(40,-2);
\draw[black, fill=blue!70!green!50] (-17,-9) rectangle +(6,7);
\draw[black] (-14,-9) -- +(0,7);
\draw[black] (-17,-5.5) -- +(6,0);
\draw[black, fill=blue!70!green!50] (-8,-9) rectangle +(6,7);
\draw[black] (-5,-9) -- +(0,7);
\draw[black] (-8,-5.5) -- +(6,0);
\draw[black,fill=blue!70!green!50] (11,-9) rectangle +(6,7);
\draw[black] (14,-9) -- +(0,7);
\draw[black] (11,-5.5) -- +(6,0);
\draw[black, fill=gray] (1,-15) rectangle +(7,14);
\draw[gray!30,fill] (7,-8) circle (0.5);
\draw[black] (-20,1) rectangle +(40,-16);
\draw[black!70, fill] (-22,1) rectangle +(44,1);
\draw[black!70, fill] (0,-15.5) rectangle +(9,1);
\end{scope}

\begin{scope}[gray, xshift=-50, yshift=-90, scale=0.3]
\draw[black,  fill=gray] (-22,2) -- +(4,15) -- +(40,15) -- +(44,0) -- cycle;
\draw[black,  fill=gray] (-8,17) -- +(0,3) -- +(4,3) -- +(4,0) -- cycle;
\draw[gray!30,top color=gray!30, bottom color=white] (-20,1) rectangle +(40,-2);
\draw[gray!30,top color=gray!30, bottom color=white] (-20,-1) rectangle +(40,-2);
\draw[gray!30,top color=gray!30, bottom color=white] (-20,-3) rectangle +(40,-2);
\draw[gray!30,top color=gray!30, bottom color=white] (-20,-5) rectangle +(40,-2);
\draw[gray!30,top color=gray!30, bottom color=white] (-20,-7) rectangle +(40,-2);
\draw[gray!30,top color=gray!30, bottom color=white] (-20,-9) rectangle +(40,-2);
\draw[gray!30,top color=gray!30, bottom color=white] (-20,-11) rectangle +(40,-2);
\draw[gray!30,top color=gray!30, bottom color=white] (-20,-13) rectangle +(40,-2);
\draw[black, fill=blue!70!green!50] (-17,-9) rectangle +(6,7);
\draw[black] (-14,-9) -- +(0,7);
\draw[black] (-17,-5.5) -- +(6,0);
\draw[black, fill=blue!70!green!50] (-8,-9) rectangle +(6,7);
\draw[black] (-5,-9) -- +(0,7);
\draw[black] (-8,-5.5) -- +(6,0);
\draw[black,fill=blue!70!green!50] (11,-9) rectangle +(6,7);
\draw[black] (14,-9) -- +(0,7);
\draw[black] (11,-5.5) -- +(6,0);
\draw[black, fill=gray] (1,-15) rectangle +(7,14);
\draw[gray!30,fill] (7,-8) circle (0.5);
\draw[black] (-20,1) rectangle +(40,-16);
\draw[black!70, fill] (-22,1) rectangle +(44,1);
\draw[black!70, fill] (0,-15.5) rectangle +(9,1);
\end{scope}

\end{tikzpicture} 
\caption{A 10-node microgrid constant power generation and constant power loads. The figure shows solar panels and residential users but the model is general enough for any type of distributed resource.}
\label{fig:test}
\end{figure}
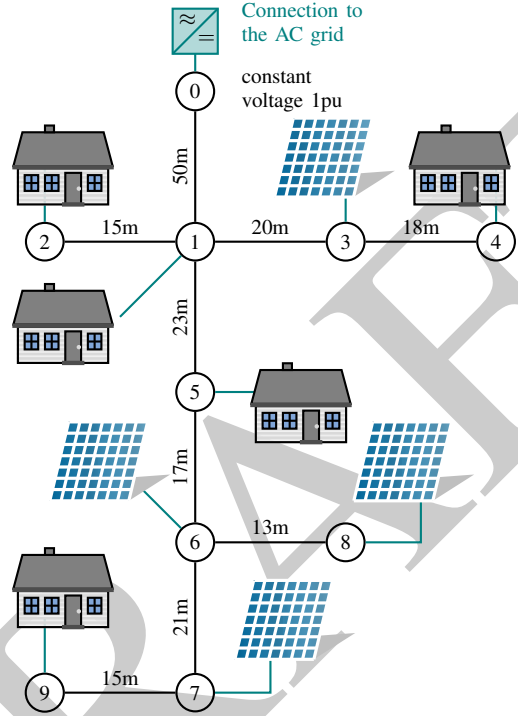

The equilibrium point was obtained by Model 1 using the Newton's method giving $v_4=  0.8903$ and $v_7=0.9427$. Theorem 1 guarantee these values are unique in the convex set $\Omega$ which defines the region of attraction. This region can be estimated by Model 2 obtaining the minimum voltages $v_{min} =[0.182,\;0.002] $ which define an open subset in $\mathbb{R}^2$ as depicted in Figure \ref{fig:roa_1}; this figure depicts also the equilibrium point and three trajectories from different initial conditions calculated by the function ode45 of Matlab. As expected, initial points inside the estimation of the region of attraction go directly to the equilibrium.  A fourth trajectory starting in the point $(0.1,0.1)$ (outside of the region of attraction) was calculated but not depicted in the figure since this trajectory was unstable and the function ode45 diverged. 

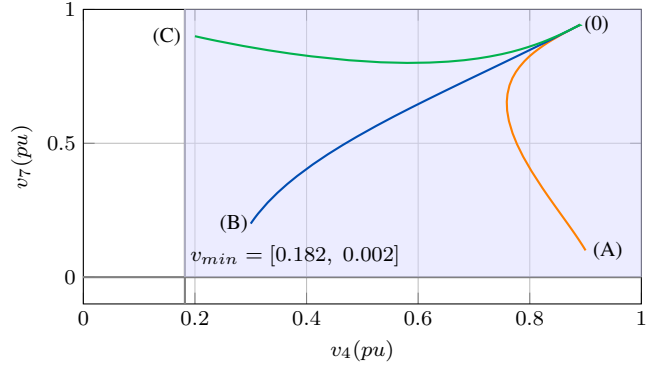
\begin{figure}[tb]
    \centering
    \footnotesize
    \begin{tikzpicture}
	\begin{axis}[xlabel=$v_4 (pu)$,ylabel=$v_7 (pu)$,width=9cm,height=5.5cm,ymin=-0.1,ymax=1,xmin=-0,xmax=1, xmajorgrids, ymajorgrids]
	\draw[gray, thick] (axis cs:-0.1,0) -- (axis cs: 1,0);
	\draw[gray, thick] (axis cs:0.182,-0.1) -- (axis cs: 0.182,1);

    \fill[blue!15, opacity=0.5] (axis cs: 0.18,0) rectangle (1,1);

	\addplot[orange,thick] coordinates{(0.9,0.1)(0.89857,0.10495)(0.89716,0.10977)(0.89575,0.11447)(0.89437,0.11906)(0.88761,0.14068)(0.88115,0.16052)(0.87496,0.17902)(0.86904,0.19647)(0.85557,0.23567)(0.84344,0.27095)(0.83252,0.3032)(0.8227,0.33298)(0.80304,0.3963)(0.7881,0.45031)(0.77706,0.49699)(0.76916,0.5378)(0.76267,0.58312)(0.75947,0.62185)(0.75883,0.65515)(0.76009,0.68404)(0.76367,0.71626)(0.76888,0.74344)(0.77518,0.7665)(0.78204,0.78627)(0.7906,0.80699)(0.79918,0.82444)(0.80757,0.83918)(0.81552,0.8518)(0.82451,0.86521)(0.83267,0.87642)(0.84005,0.88579)(0.84659,0.89373)(0.85374,0.90237)(0.85979,0.90942)(0.86495,0.91511)(0.86927,0.9198)(0.87379,0.92491)(0.87737,0.92886)(0.88023,0.93184)(0.88246,0.93416)(0.8846,0.93667)(0.88618,0.93844)(0.8874,0.93954)(0.88827,0.94033)(0.88857,0.94152)(0.88893,0.94215)(0.88983,0.9419)}; 
	
	\addplot[blue!70!green,thick] coordinates{(0.3,0.2)(0.30327,0.20986)(0.30656,0.21938)(0.30989,0.22859)(0.31325,0.23752)(0.33035,0.27862)(0.34781,0.31497)(0.36543,0.34772)(0.38306,0.37763)(0.41109,0.42087)(0.43848,0.45926)(0.46501,0.49372)(0.4905,0.52497)(0.5251,0.56524)(0.55725,0.60085)(0.58694,0.63254)(0.61422,0.66096)(0.64764,0.69524)(0.67717,0.72509)(0.70321,0.75115)(0.72609,0.77401)(0.7538,0.80179)(0.77681,0.82495)(0.79587,0.84429)(0.8117,0.86045)(0.82947,0.87852)(0.84317,0.89264)(0.85354,0.90373)(0.86158,0.91236)(0.86842,0.91935)(0.87361,0.92475)(0.87742,0.92896)(0.88035,0.93218)(0.8832,0.93498)(0.88522,0.93705)(0.88649,0.93866)(0.88742,0.93981)(0.88864,0.94045)(0.88935,0.941)(0.88932,0.94182)(0.88934,0.94234)};
	
	\addplot[green!70!blue,thick] coordinates {(0.2,0.9)(0.21001,0.89535)(0.21994,0.89086)(0.22978,0.88652)(0.23951,0.88233)(0.28646,0.86348)(0.33029,0.84786)(0.37086,0.83511)(0.40823,0.82487)(0.47597,0.81023)(0.53279,0.80258)(0.58029,0.8002)(0.62025,0.80147)(0.65386,0.80504)(0.68244,0.81033)(0.70681,0.8168)(0.72774,0.8239)(0.74962,0.83285)(0.76803,0.84189)(0.78354,0.85078)(0.79678,0.85927)(0.81049,0.86875)(0.82195,0.87742)(0.83151,0.88533)(0.8396,0.89241)(0.84823,0.90008)(0.85529,0.90666)(0.86103,0.91235)(0.86579,0.91718)(0.87091,0.92224)(0.87493,0.92632)(0.87802,0.92965)(0.88047,0.93231)(0.88311,0.93494)(0.88503,0.93691)(0.88631,0.93845)(0.88727,0.93958)(0.88846,0.94038)(0.88918,0.94101)(0.88929,0.94175)(0.8894,0.94221)(0.88992,0.94209)(0.89019,0.94212)(0.88999,0.94249)(0.88988,0.94274)};
	
	\node at (axis cs: 0.38,0.08) {$v_{min}=[0.182,\;0.002]$};
	\node at (axis cs:0.92, 0.9427) {(0)};
	\node at (axis cs:0.94, 0.1) {(A)};
	\node at (axis cs:0.27, 0.2) {(B)};
	\node at (axis cs:0.15, 0.9) {(C)};
	
	\end{axis}
\end{tikzpicture}
    \caption{Estimation of the region of attraction for the case of two constant power loads.  Three different trajectories starting from points (A), (B) and (C) converge to the equilibrium point (0)}
    \label{fig:roa_1}
\end{figure}

Let us consider now the general case with 7 constant power terminals: 3 constant power generations (nodes 3,7,8) and 4 constant power loads (nodes 2,3,5,9).  The main results of each model are given in Table \ref{tab:parametros}.  Let us analyze each one of them.

The equilibrium point and the estimation of the region of attraction were obtained by Models 1 and 2 respectively. The estimation of the region of attraction is a hypercuve with the same interpretation as in the previous case.  Now, this region can be used as criteria for voltage protection.  \textcolor{black}{In order to evaluate the under-voltage protection criteria, let us consider the case of a voltage sag in the main converter;  the grid starts from an initial condition of $v_0=1pu$ but the voltage is reduced until $v_0=0.2pu$ for $t=0.05s$. As result, nodal voltages reach values bellow $v_{min}$ and the system becomes unstable (instability is diagnosed numerically since the function ode45 of Matlab diverge)}. This problem was avoided by disconnecting terminals with values below $v_{min}$.  Figure \ref{fig:protection} shows this type of protection for Load $p_2$ (similar behaviour is obtained for the rest of the terminals).  Notice that the region of attraction was calculated considering the entire grid, but the criteria for the under-voltage protection is completely local, meaning that only its own measure of voltage is required.

\begin{figure}[tb]
    \centering
    \footnotesize
    \begin{tikzpicture}
	\begin{axis}[xlabel=$t(pu)$,ylabel=$p_2(pu)$,width=9cm,height=4cm,ymin=0,ymax=1,xmin=0,xmax=0.1, xmajorgrids, ymajorgrids]
	\addplot[blue!70!green,thick] coordinates{(0,1)(0.0001,0.8926)(0.0002,0.80329)(0.0003,0.73919)(0.0004,0.68891)(0.0005,0.64611)(0.0006,0.60836)(0.0007,0.57462)(0.0008,0.54334)(0.0009,0.51456)(0.001,0.48767)(0.0011,0.4622)(0.0012,0.43841)(0.0013,0.41559)(0.0014,0.39388)(0.0015,0.3732)(0.0016,0.35312)(0.0017,0.33397)(0.0018,0.31535)(0.0019,0.29723)(0.002,0.27975)(0.0021,0.26246)(0.0022,0.2456)(0.0023,0.22892)(0.0024,0.21224)(0.0025,0.19565)(0.0026,0.17867)(0.0027,0.1612)(0.0028,0.14276)(0.0029,0.12238)(0.003,0.098287)(0.0031,0.095782)(0.0032,0.095743)(0.0033,0.095625)(0.0034,0.095677)(0.0035,0.095622)(0.0036,0.095704)(0.0037,0.095698)(0.0038,0.095696)(0.0039,0.095687)(0.004,0.095678)(0.005,0.085158)(0.01,0.085066)(0.02,0.085066)(0.03,0.085067)(0.04,0.085067)(0.05,0.085066)(0.051,0.50186)(0.052,0.68065)(0.053,0.78989)(0.054,0.85947)(0.055,0.90422)(0.056,0.93291)(0.057,0.95129)(0.058,0.9632)(0.059,0.97105)(0.06,0.97614)(0.061,0.97925)(0.062,0.98118)(0.063,0.98256)(0.064,0.98358)(0.065,0.98418)(0.066,0.98441)(0.067,0.98453)(0.068,0.98478)(0.069,0.98502)(0.07,0.98506)(0.08,0.98508)(0.09,0.98499)(0.1,0.98505)};
	\end{axis}
	\end{tikzpicture}
    \begin{tikzpicture}
	\begin{axis}[xlabel=$t(pu)$,ylabel=$v_2(pu)$,width=9cm,height=4cm,ymin=-80,ymax=5,xmin=0,xmax=0.1, xmajorgrids, ymajorgrids]	
	\addplot[blue!70!green,thick] coordinates{(0,-70)(0.003,-70.8352)(0.003,0)(0.05,0)(0.051,-70.4981)(0.052,-70.3192)(0.053,-70.21)(0.1,-70.015)};
	\end{axis}
\end{tikzpicture}
    \caption{Using the region of attraction as indicator for under-voltage protection.  Due to the voltage sag in $v_0$ the voltage in node 2 is reduced until reach a value below that $v_{min}$.  The system disconnects the load until the voltages are again feasible.  Stability is guaranteed in the entire process }
    \label{fig:protection}
\end{figure}
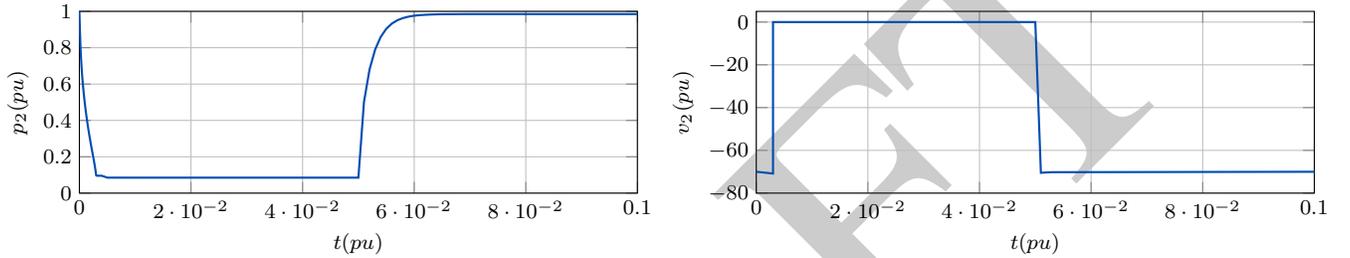

\textcolor{black}{These numerical simulations are available in \cite{yomatlab} in order to reproduce the results and evaluate other possible configurations. }

\section{Conclusions}

A convex based methodology for transient stability analysis of dc-micrids was presented.  This methodology was based on a gradient representation of the grid and the general observation that the energy function is strongly-convex.  The dynamical problem was then transformed into \textcolor{black}{optimization problems in which Model 1 allowed to find the equilibrium point and Model 2 estimated the region of attraction. The later is also used in practice as a criteria for under-voltage protection.}
The main result was demonstrated mathematically using the Lyapunov method under well defined consideration regarded the convexity of the Lyapunov function.  A complete set of simulations was performed in order demonstrate the practical applicability of the methodology.

\bibliographystyle{IEEEtran}
\bibliography{bibliografia}

\begin{thebibliography}{10}
\providecommand{\url}[1]{#1}
\csname url@samestyle\endcsname
\providecommand{\newblock}{\relax}
\providecommand{\bibinfo}[2]{#2}
\providecommand{\BIBentrySTDinterwordspacing}{\spaceskip=0pt\relax}
\providecommand{\BIBentryALTinterwordstretchfactor}{4}
\providecommand{\BIBentryALTinterwordspacing}{\spaceskip=\fontdimen2\font plus
\BIBentryALTinterwordstretchfactor\fontdimen3\font minus
  \fontdimen4\font\relax}
\providecommand{\BIBforeignlanguage}[2]{{%
\expandafter\ifx\csname l@#1\endcsname\relax
\typeout{** WARNING: IEEEtran.bst: No hyphenation pattern has been}%
\typeout{** loaded for the language `#1'. Using the pattern for}%
\typeout{** the default language instead.}%
\else
\language=\csname l@#1\endcsname
\fi
#2}}
\providecommand{\BIBdecl}{\relax}
\BIBdecl

\bibitem{review_dcmicrogirds}
L.~Meng, Q.~Shafiee, G.~F. Trecate, H.~Karimi, D.~Fulwani, X.~Lu, and J.~M.
  Guerrero, ``Review on control of dc microgrids and multiple microgrid
  clusters,'' \emph{IEEE Journal of Emerging and Selected Topics in Power
  Electronics}, vol.~5, no.~3, pp. 928--948, Sept 2017.

\bibitem{adhoc}
W.~Inam, J.~A. Belk, K.~Turitsyn, and D.~J. Perreault, ``Stability, control,
  and power flow in ad hoc dc microgrids,'' in \emph{2016 IEEE 17th Workshop on
  Control and Modeling for Power Electronics (COMPEL)}, June 2016, pp. 1--8.

\bibitem{DC_Parte1}
T.~Dragicevic, X.~Lu, J.~C. Vasquez, and J.~M. Guerrero, ``Dc microgrids part
  i: A review of control strategies and stabilization technique,'' \emph{IEEE
  Transactions on Power Electronics}, no.~7, pp. 4876--4891, July.

\bibitem{koguiman}
S.~Sanchez, R.~Ortega, R.~Grino, G.~Bergna, and M.~Molinas, ``Conditions for
  existence of equilibria of systems with constant power loads,'' \emph{IEEE
  Transactions on Circuits and Systems I: Regular Papers}, vol.~61, no.~7, pp.
  2204--2211, July 2014.

\bibitem{romeo_existencia_equilibrio}
N.~Barabanov, R.~Ortega, R.~Grino, and B.~Polyak, ``On existence and stability
  of equilibria of linear time-invariant systems with constant power loads,''
  \emph{IEEE Transactions on Circuits and Systems I: Regular Papers}, vol.~63,
  no.~1, pp. 114--121, Jan 2016.

\bibitem{yo_elsevier}
\BIBentryALTinterwordspacing
A.~Garces, ``Uniqueness of the power flow solutions in low voltage direct
  current grids,'' \emph{Electric Power Systems Research}, vol. 151, pp. 149 --
  153, 2017. [Online]. Available:
  \url{http://www.sciencedirect.com/science/article/pii/S0378779617302298}
\BIBentrySTDinterwordspacing

\bibitem{yo_tps}
------, ``On the convergence of newton's method in power flow studies for dc
  microgrids,'' \emph{IEEE Transactions on Power Systems}, vol.~33, no.~5, pp.
  5770--5777, Sept 2018.

\bibitem{small}
R.~Majumder, ``Some aspects of stability in microgrids,'' \emph{IEEE
  Transactions on Power Systems}, vol.~28, no.~3, pp. 3243--3252, Aug 2013.

\bibitem{small2}
A.~A.~A. Radwan and Y.~A.~I. Mohamed, ``Linear active stabilization of
  converter-dominated dc microgrids,'' \emph{IEEE Transactions on Smart Grid},
  vol.~3, no.~1, pp. 203--216, March 2012.

\bibitem{small3}
Z.~Li and M.~Shahidehpour, ``Small-signal modeling and stability analysis of
  hybrid ac/dc microgrids,'' \emph{IEEE Transactions on Smart Grid}, pp. 1--1,
  2018.

\bibitem{transient_review}
M.~Kabalan, P.~Singh, and D.~Niebur, ``Large signal lyapunov-based stability
  studies in microgrids: A review,'' \emph{IEEE Transactions on Smart Grid},
  vol.~8, no.~5, pp. 2287--2295, Sept 2017.

\bibitem{stab}
M.~Su, Z.~Liu, Y.~Sun, H.~Han, and X.~Hou, ``Stability analysis and
  stabilization methods of dc microgrid with multiple parallel-connected dc-dc
  converters loaded by cpls,'' \emph{IEEE Transactions on Smart Grid}, vol.~9,
  no.~1, pp. 132--142, Jan 2018.

\bibitem{cpl1}
J.~Liu, W.~Zhang, and G.~Rizzoni, ``Robust stability analysis of dc microgrids
  with constant power loads,'' \emph{IEEE Transactions on Power Systems},
  vol.~33, no.~1, pp. 851--860, Jan 2018.

\bibitem{estabilidad_sdp}
L.~Herrera, W.~Zhang, and J.~Wang, ``Stability analysis and controller design
  of dc microgrids with constant power loads,'' \emph{IEEE Transactions on
  Smart Grid}, vol.~8, no.~2, pp. 881--888, March 2017.

\bibitem{kron}
G.~Kron, \emph{Tensors for circuits}.\hskip 1em plus 0.5em minus 0.4em\relax
  Dover Publications, 1942.

\bibitem{danilo}
O.~D. Montoya, ``Numerical approximation of the maximum power consumption in
  dc-mgs with cpls via an sdp model,'' \emph{IEEE Transactions on Circuits and
  Systems II: Express Briefs}, pp. 1--1, 2018.

\bibitem{hvdc}
A.~Doria-Cerezo, J.~M. Olm, M.~di~Bernardo, and E.~Nuno, ``Modelling and
  control for bounded synchronization in multi-terminal vsc-hvdc transmission
  networks,'' \emph{IEEE Transactions on Circuits and Systems I: Regular
  Papers}, vol.~63, no.~6, pp. 916--925, June 2016.

\bibitem{nemirovski}
Y.~Nesterov and A.~Nemirovskii, \emph{Interior point polynomial algorithms in
  convex programming}, 1st~ed., ser. 10.\hskip 1em plus 0.5em minus 0.4em\relax
  Philadelphia: SIAM studies in applied mathematics, 1994, vol.~1.

\bibitem{yo_doscpl}
A.~Garces and A.~Gutierrez, ``On the stability of dc microgrids with two
  constant power devices,'' in \emph{2018 IEEE Green Technologies Conference
  (GreenTech)}, April 2018, pp. 33--37.

\bibitem{yomatlab}
\BIBentryALTinterwordspacing
A.~Garces. (2016, May 20) Matlab central file exchange. [Online]. Available:
  \url{http://www.mathworks.com/matlabcentral/
  profile/authors/3009175-alejandro-garces}
\BIBentrySTDinterwordspacing

\end{thebibliography}
\end{document}